\documentclass[11pt]{amsart}
\usepackage{amsmath,amsthm,epsfig,amssymb}
\usepackage{amsthm,amsfonts,amssymb,amscd}
\usepackage[utf8x]{inputenc}
\usepackage{float}
\usepackage{textcomp}
\usepackage{tikz}
\usetikzlibrary{positioning, matrix, arrows}
\usepackage{tikz-cd}
\usepackage[all]{xy}
\usepackage[hyphens]{url}
\usepackage[colorlinks=true,citecolor=blue,urlcolor=blue,linkcolor=blue]{hyperref}
\usepackage{array}
\usepackage{multirow}
\usepackage{enumerate}
\usepackage{graphicx}
\usepackage{mathrsfs}
\usepackage{afterpage}
\usepackage{lmodern}

\newtheorem{Theorem}[subsection]{Theorem}

\newtheorem{Lemma}[subsection]{Lemma}

\theoremstyle{definition}
\newtheorem{Remark}[subsection]{Remark}

\newtheorem{proposition-definition}{Proposition-Definition}

\title{Kernel Sheaf on Integral Nodal Curves}
\author[Basu, Dan, Rath]{Suratno Basu $^{*} $, Krishanu Dan $^{**} $, Aanjaneya Rath$^{***}$ }

\subjclass[2020]{14C20, 14D20, 14H40}

\keywords{Nodal Curve, Kernel Sheaf, Stable Sheaf \\
$^{*}$
Institute for advancing intelligence (IAI), TCG, CREST, Academy of Scientific and innovative Research, (ACSIR), Kolkata, India.
Email: suratno.basu@tcgcrest.org\\
$^{**}$  
School of Mathematical Sciences, National Institute of Science Education and Research,
Bhubaneswar, An OCC of Homi Bhabha National Institute,
Khurda 752050, Odisha, India.
Email: krishanu@niser.ac.in\\
$^{***}$  
School of Mathematical Sciences, National Institute of Science Education and Research,
Bhubaneswar, An OCC of Homi Bhabha National Institute,
Khurda 752050, Odisha, India.
Email: aanjaneya.rath@niser.ac.in }

\begin{document}

\begin{abstract}
In this article we study the stability of Kernel sheaf obtained from a generating subspace of rank one torsion-free sheaf on an integral nodal curve.
\end{abstract}

\maketitle

\section{Introduction}

Let $(Y, H)$ be a polarized smooth projective variety over $\mathbb{C}$ and let $E$ be a globally generated $\mu_H$-semistable vector bundle on $Y$. Let $V \subseteq H^0(Y, E)$ be a generating subspace. Then we have an exact sequence
$$
0 \longrightarrow M_{V, E} \longrightarrow V \otimes \mathcal{O}_Y \longrightarrow E \longrightarrow 0.
$$
Then vector bundle $M_{V, E}$ is called the {\it syzygy bundle} or {\it kernel bundle}. When $Y$ is a curve, $E = K_Y$ and $V = H^0(Y, K_Y)$, the stability of $M_{V, E}$ was studied by Paranjape-Ramanan (\cite{P-R}). In \cite{Butler}, Butler studied the stability of $M_{V, E}$ when $X$ is a curve and $V = H^0(Y, E)$. Ein-Lazarsfeld-Mustopa (\cite{E-L-M}) proved for stability of $M_{V, E}$ when $X$ is a surface and $E = dH$, for some $d \gg 0$. Recently, Rekuski (\cite{R}) prove the stability when $E$ is a line bundle and $V = H^0(Y, E)$.

Now assume $Y$ is a curve. Using the Brill-Noether loci and moduli of coherent systems, Butler conjecture for ``general" choice of $(L, V)$ where $L$ is a globally generated line bundle on $Y$ and $V \subseteq H^0(Y, L)$ is a generating set, $M_{V, L}$ is stable. In fact this was stated for more general vector bundles, see \cite[Conjecture 2]{Butler-2} for a precise statement. In \cite{Bhosle-3}, the authors gave a proof of Butler conjecture for a line bundles. On the other hand, Mistretta in \cite{Mistretta} studied stability of $M_{V, L}$ for generating subspaces $V \subseteq H^0(Y, L)$ of certain codimension.

Let $X$ be an integral projective curve of genus $g \geq 2$ over $\mathbb{C}$ with only ordinary nodes as singularities. Let $E$ be a globally generated torsion-free sheaf on $X$ and $V \subseteq H^0(X, E)$ be a generating subspace. Then, there is a torsion-free sheaf $M_{V, E}$ on $X$ such that we have an exact sequence
$$
0 \longrightarrow M_{V, E}^* \longrightarrow V \otimes \mathcal{O}_X \longrightarrow E \longrightarrow 0.
$$
The torsion free sheaf $M_{V, E}^*$ is called the {\it kernel sheaf} on $X$ associated to the generating subspace $V \subseteq H^0(X, E)$. In \cite{Bhosle-2}, the authors studied the stability of $M_{V, E}^*$ when $V = H^0(X, E)$ and use it in the study of Brill-Noether loci. The notion of kernel sheaf can also be generalized to the case of reducible nodal curve. In \cite{B-F}, the authors studied stability of kernel sheaves over reducible nodal curves.

Now let $X$ be an integral projective curve of genus $g \geq 2$ with only ordinary nodes as singularities. In this article, we have studied the stability of $M_{V, L}^*$ where $L$ is a rank one torsion-free and $V \subseteq H^0(X, L)$ is a generating subspace of certain codimension. When $X$ has one node, we got the following

\begin{Theorem}\label{thm 1-1}
Let $X$ be an integral projective curve with an ordinary double point, $g \geq 2$ be the arithmetic genus of $X$ and let $L$ be a globally generated torsion-free sheaf of rank one and degree $d > 2g + 2c, 1 \leq c \leq g$. Then for a generic subspace $V \subseteq H^0(X, L)$ of codimension $c$ that generates $E$, $M_{V, L}^*$ is stable.
\end{Theorem}

When $d = 2g + 2c$, we have shown (Theorem \ref{thm 3-2}) that $M_{V, L}^*$ is semistable and it is stable when $X$ is non-hyperelliptic. We have also studied the case of multi-nodes. In this case we  obtained the following

\begin{Theorem}\label{thm 1-2}
Let $X$ be an integral projective curve of arithmetic genus $g \geq 3$ with $n$-nodes. Let $L$ be a globally generated torsion-free sheaf of rank one and degree $d > 10(g+c)/3, 1 \leq c \leq g$. Then for a generic subspace $V \subseteq H^0(X, L)$ of codimension $c$ that generates $L$, $M_{V, L}^*$ is stable. Moreover, if the number of node $n \leq g - 2$, then $M_{V, L}^*$ is stable for $d > 2g + 2c$.
\end{Theorem}

Our approach is a suitable adaptation of the arguments as given in \cite{Mistretta}: we construct a parameter space that contains data of all possible destabilization of $M_{V, L}^*$ and show that dimension of this space is less than the dimension of $\text{Gr}(c, H^0(C, L))$. Then for a generic choice of $V \in \text{Gr}(c, H^0(C, L))$ no such destabilization exists and hence $M_{V, L}^*$ is stable.

\section{Kernel Sheaf}

Let $X$ be an integral projective curve with ordinary double points as singularities and let $g$ be the arithmetic genus of $X$. 

\begin{Lemma}\label{lem 2-1}
Let $0 \longrightarrow E_1 \longrightarrow E_2 \longrightarrow E_3 \longrightarrow 0$ be a short exact sequence of torsion-free sheaves on $X$. Then
\begin{itemize}
\item[$(i)$] The sequence $0 \longrightarrow E_3^* \longrightarrow E_2^* \longrightarrow E_1^* \longrightarrow 0$ is exact;

\item[$(ii)$] if any two of them is locally free, then so is the third.
\end{itemize}
\end{Lemma}
\begin{proof}
See \cite[Lemma 2.2]{Bhosle-1}
\end{proof}

Let $E$ be a globally generated torsion-free  sheaf on $X$ and $V \subseteq H^0(X, E)$ be a generating subspace. Then there is a torsion-free sheaf $M_{V, E}$ on $X$ such that we have a commutative diagram:
$$
0 \longrightarrow M_{V, E}^* \longrightarrow V \otimes \mathcal{O}_X \longrightarrow E \longrightarrow 0.
$$
By Lemma \ref{lem 2-1}, if $E$ is locally free, then so is $M_{V, E}^*$.

\begin{Lemma}\label{lem 2-2}
Let $E$ be a globally generated semistable torsion-free sheaf on $X$. Let $V \subseteq H^0(X, E)$ be a subspace that generates $E$ and let $N$ be a stable saturated subsheaf of $M_{V, E}^*$ of maximal slope. Then there is a torsion-free sheaf $F$ with no trivial summand, a subspace $W \subseteq H^0(X, F)$ generating $F$ and a non-zero morphism $\sigma: F \longrightarrow E$ such that $N = M_{W, F}^*$. Moreover, $\mu(F) \leq \mu\big( \sigma(F) \big)$ and in case of equality, $F$ is a subsheaf of $E$.
\end{Lemma}
\begin{proof}
We will argue as in \cite[Lemma 2.2]{Mistretta} (see also \cite[Lemma 3.9]{Bhosle-2}). By Lemma \ref{lem 2-1}, $V^*\otimes \mathcal{O}_X^* \longrightarrow M_{V, E}^{**} \longrightarrow N^*$ is surjective and $N^*$ is generated by sections. Let $W^* = \text{Im}\big( V^* \longrightarrow H^0(X, N^*) \big)$. Then we have a short exact sequence $0 \longrightarrow F_1 \longrightarrow W^* \otimes \mathcal{O}_X^* \longrightarrow N^* \longrightarrow 0$ of torsion-free sheaves on $X$. Taking dual again and using Lemma \ref{lem 2-1}, we get a short exact sequence $0 \longrightarrow N^{**} \longrightarrow W \otimes \mathcal{O}_X \longrightarrow F_1^* \longrightarrow 0$. Since $N$ is torsion-free, we may consider it as a subsheaf of $N^{**}$. Since $\big( M_{V, E}^{**} \big)^* = M_{V, E}^*$ and $N^{**}$ is a subsheaf of $\big( M_{V, E}^{**} \big)^*$, we get $N = N^{**}$ and we have a commutative diagram
 \begin{center}
 \begin{tikzpicture}
  \matrix(m)[matrix of math nodes,row sep=2.0em, column sep=3.5em,text height=1.5ex, 
 text depth=0.25ex]
 {  0 &  M_{V, E}^* & V \otimes \mathcal{O}_X & E & 0\\
 0 & N & W \otimes \mathcal{O}_X & F & 0 \\};
 \path[->](m-1-1) edge (m-1-2);
 \path[->](m-1-2) edge  (m-1-3);
 \path[->](m-1-3) edge  (m-1-4);
 \path[->](m-1-4) edge  (m-1-5);
 \path[->](m-2-1) edge (m-2-2);
 \path[->](m-2-2) edge  (m-2-3);
 \path[->](m-2-3) edge  (m-2-4);
 \path[->](m-2-4) edge  (m-2-5);
 \path[right hook ->](m-2-2) edge  (m-1-2);
 \path[right hook ->](m-2-3) edge  (m-1-3);
 \path[->](m-2-4) edge node[right]{$\sigma$}  (m-1-4);
 \end{tikzpicture}
 \end{center}
where $F = F_1^*$ is a reflexive sheaf on $X$ and $\sigma$ is a non-zero morphism. Let $I = \sigma(F) \subseteq E$. Then we have a commutative diagram of vector spaces
\begin{center}
 \begin{tikzpicture}
  \matrix(m)[matrix of math nodes,row sep=2.0em, column sep=3.5em,text height=1.5ex, 
 text depth=0.25ex]
 {  V &   & H^0(X, E)\\
 W &  H^0(X, F) & H^0(X, I) \\};
 \path[right hook ->](m-1-1) edge (m-1-3);
 \path[right hook ->](m-2-1) edge (m-2-2);
 \path[->](m-2-2) edge  (m-2-3);
 \path[right hook ->](m-2-1) edge  (m-1-1);
 \path[right hook ->](m-2-3) edge  (m-1-3);
 \end{tikzpicture}
 \end{center}
Thus the map $W \longrightarrow H^0(X, I)$ is injective and we denote its image by $W$, again. Then we have a commutative diagram with exact rows
\[\xymatrixcolsep{3.5pc}\xymatrixrowsep{2.5pc}\xymatrix@R-1pc{
  0\ar[r] & M_{V, E}^* \ar[r] & V \otimes \mathcal{O}_X \ar[r] & E\ar[r] & 0\\
  0\ar[r] & M_{W, I}^*\ar@{^{(}->}[u] \ar[r] & W \otimes \mathcal{O}_X \ar@{^{(}->}[u] \ar[r] & I\ar@{^{(}->}[u] \ar[r] & 0\\
  0\ar[r] & N\ar@{^{(}->}[u] \ar[r] & W \otimes \mathcal{O}_X\ar@{=}[u] \ar[r] & F \ar@{->>}[u]_{\sigma} \ar[r] & 0
 }\]
By the condition on $N$, we have
\begin{align*}
\mu(N) \geq \mu(M_{W, I}^*) \,\, &\Longrightarrow \,\, - \big( \deg F/ \text{rk}N \big) \geq - \big( \deg I/\text{rk} M_{W, I}^* \big)\\
&\Longrightarrow \,\, \deg F \leq \big( \text{rk}N/\text{rk} M_{W, I}^* \big) \deg I \leq \deg I.
\end{align*}

Suppose $\text{rk} F = \text{rk} I$. Then $\text{rk} N = \text{rk} M_{W, I}^*$. If $\deg F = \deg I$. then $F \cong I$. In case of $\deg F < \deg I$, by semitability of $E$ we get $\deg F < \deg I \leq \deg E$ so that $\mu(F) < \mu(I)$.

Now assume $\text{rk} I < \text{rk} F$. Then by semitability of $E$ we have
\begin{align*}
\deg F \cdot \text{rk} I < \deg I \cdot \text{rk} F \,\, \Longrightarrow \,\, \mu(F) < \mu(I) \leq \mu(E).
\end{align*}
If $\mu(F) = \mu(I)$, then by above we must have $\text{rk}(I) = \text{rank}(F)$ and $\deg I = \deg F$. Hence $F \cong I$.
\end{proof}

\begin{Lemma}\label{lem 2-3}
Let $E$ be a semistable torsion-free sheaf on $X$ with no trivial summand of rank $r$ and degree $d \geq 2gr + 2c$ where $1 \leq c\leq g$. Let $V \subseteq H^0(X, E)$ be a subspace of codimension $c$ such that $V$ generates $E$ and let $N$ be a stable subbundle of $M_{V, E}^*$ of maximal slope. Let $F$ be the torsion-free sheaf as obtained in Lemma \ref{lem 2-2}. Suppose $h^1(X, F) = 0$. Then  $\text{rk} F \leq r$, $\deg F \leq d$ and $\mu(F) \leq \mu(E)$. Moreover, if $E$ is stable and $\mu(E) = \mu(F)$, then $F \cong E$.
\end{Lemma}
\begin{proof}
We have
$$
\dim W \leq h^0(X, F)  = \deg F + \text{rk}F (1 - g) = \text{rk}F ( \mu(F) + 1- g)
$$
so that
$$
\mu(N) = \dfrac{- \deg F}{\dim W - \text{rk} F} \leq  \dfrac{-\mu(F)}{\mu(F) - g}.
$$
On the other hand, $\text{rk} M_{V, E}^* = \dim V - r = h^0(X, E) - c - r = d - rg - c.$ If $\text{rk} F \geq r+ 1$, then $\mu(F) \leq d/(r+1)$ and we have
\begin{equation}\label{E1}
\mu(N) \leq \dfrac{-\mu(F)}{\mu(F) - g} \leq  \dfrac{-d}{d - (r+1) g} \leq \dfrac{-d}{d - rg - c} = \mu(M_{V, E}^*)
\end{equation}
where the second inequality follows from the fact that $-x/(x-g)$ is an increasing function for $x > g$. If $\mu(N) = \mu(M_{V, E}^*)$, then all the inequality in \eqref{E1} are equalities. From the third equality, we get $g = c$, from the first we get $W = H^0(X, F)$, from the second and Lemma \ref{lem 2-2}, we get $\deg F = d$ and $\text{rk} F =  r+ 1$. But then we have
\begin{align*}
h^0(X, F) = d + (r+1)(1-g) = h^0(X, E) + 1 - g > h^0(X, E).
\end{align*}
This contradicts that fact that $h^0(X, F) = \dim W \leq \dim V \leq h^0(X, E)$. Hence $\text{rk} F \leq r = \text{rk} E$. By Lemma \ref{lem 2-2}, $\mu(F) \leq \mu(E)$.
\end{proof}

\vspace{0.2 cm}

\begin{Remark}
Let $L$ be a torsion-free sheaf on $X$ of rank one and degree $d \geq 2g + 2c$, $V \subseteq H^0(X, L)$ be a generating subspace of codimension $c, 1 \leq c \leq g$ and $N$ be a stable subbundle of $M_{V, L}^*$ of maximal slope. Consider the following diagram as obtained in Lemma \ref{lem 2-2}:
 \begin{center}
 \begin{tikzpicture}
  \matrix(m)[matrix of math nodes,row sep=2.0em, column sep=3.5em,text height=1.5ex, 
 text depth=0.25ex]
 {  0 &  M_{V, L}^* & V \otimes \mathcal{O}_X & L & 0\\
 0 & N & W \otimes \mathcal{O}_X & F & 0 \\};
 \path[->](m-1-1) edge (m-1-2);
 \path[->](m-1-2) edge  (m-1-3);
 \path[->](m-1-3) edge  (m-1-4);
 \path[->](m-1-4) edge  (m-1-5);
 \path[->](m-2-1) edge (m-2-2);
 \path[->](m-2-2) edge  (m-2-3);
 \path[->](m-2-3) edge  (m-2-4);
 \path[->](m-2-4) edge  (m-2-5);
 \path[right hook ->](m-2-2) edge  (m-1-2);
 \path[right hook ->](m-2-3) edge  (m-1-3);
 \path[->](m-2-4) edge node[right]{$\sigma$}  (m-1-4);
 \end{tikzpicture}
 \end{center}
Then by Lemma \ref{lem 2-3}, $F$ is a torsion-free sheaf of rank one and $\sigma$ is an inclusion. Let $b := \text{codim}_{H^0(X, F)}(W)$ and $s := d - \deg F$. Then by \cite[Remark 2.4]{Mistretta}, we get
\begin{equation}\label{E2}
0 < c -b < s \leq \dfrac{d}{g + c}(c - b)
\end{equation}
In particular, $s \geq 2$. For the rest of the article, we will use these notations.
\end{Remark}

\section{Stability of Kernel Sheaf: Single Node}

Let $X$ be an integral projective curve with an ordinary double point and let $g (\geq 2)$ be the arithmetic genus of $X$. Let $\pi: Y \longrightarrow X$ be the normalization of $X$. Then the genus of $Y$ is $g_Y = g - 1$.

\begin{Lemma}\label{lem 3-1}
Let $F, L$ be two rank one torsion-free sheaves on $X$ such that $\mathcal{H}om_{\mathcal{O}_X}(F, L)$ is a line bundle on $X$. Then both $F$ and $L$ are locally free.
\end{Lemma}
\begin{proof}
Let $x \in X$ be the node. Write $A  = \mathcal{O}_{X, x}$ and $\mathfrak{m} \subset A$ be the unique maximal ideal. Then $F_x \cong A$ or $F_x \cong \mathfrak{m}$ (resp. $L_x \cong A$ or $L_x \cong \mathfrak{m}$) depending on $F$ (resp. $L$) is locally free or not. Thus we have four choices
$$
\text{Hom}_A(A, A), \, \text{Hom}_A(\mathfrak{m}, A), \, \text{Hom}_A(A, \mathfrak{m}), \, \text{Hom}_A(\mathfrak{m}, \mathfrak{m})
$$
for the free $A$-module $\mathcal{H}om_{\mathcal{O}_X}(F, L)_x \cong \text{Hom}_A(F_x, L_x)$. By \cite[Chapter 8, Lemma 4]{Seshadri}, only $\text{Hom}_A(A, A)$ is the free $A$-module. Hence both $F$ and $L$ are locally free.

\end{proof}

\begin{Theorem}\label{thm 3-1}
Let $X$ be an integral projective curve with an ordinary double point, $g \geq 2$ be the arithmetic genus of $X$ and let $L$ be a globally generated torsion-free sheaf of rank one and degree $d > 2g + 2c, 1 \leq c \leq g$. Then for a generic subspace $V \subseteq H^0(X, L)$ of codimension $c$ that generates $L$, $M_{V, L}^*$ is stable.
\end{Theorem}
\begin{proof}
Let $N$ be a saturated stable sub-bundle of $M_{V, L}^*$ of maximal slope with $\mu(N) \geq \mu(M_{V, L}^*)$. Consider the following commutative diagram as obtained in Lemma \ref{lem 2-2}:
 \begin{center}
 \begin{tikzpicture}
  \matrix(m)[matrix of math nodes,row sep=2.0em, column sep=3.5em,text height=1.5ex, 
 text depth=0.25ex]
 {  0 &  M_{V, L}^* & V \otimes \mathcal{O}_X & L & 0\\
 0 & N & W \otimes \mathcal{O}_X & F & 0 \\};
 \path[->](m-1-1) edge (m-1-2);
 \path[->](m-1-2) edge  (m-1-3);
 \path[->](m-1-3) edge  (m-1-4);
 \path[->](m-1-4) edge  (m-1-5);
 \path[->](m-2-1) edge (m-2-2);
 \path[->](m-2-2) edge  (m-2-3);
 \path[->](m-2-3) edge  (m-2-4);
 \path[->](m-2-4) edge  (m-2-5);
 \path[right hook ->](m-2-2) edge  (m-1-2);
 \path[right hook ->](m-2-3) edge  (m-1-3);
 \path[->](m-2-4) edge node[right]{$\sigma$}  (m-1-4);
 \end{tikzpicture}
 \end{center}
Since $d > 2g + 2c$, $h^1(X, L) = 0$. So from Riemann-Roch, we get
$$
\dim V = h^0(X, L) - c = d + 1 - g -c.
$$
Then
$$
\mu(M_{V, L}^*) = \dfrac{\deg L}{\dim V - 1} = \dfrac{-d}{d - g - c} > -2
$$
since $d > 2g + 2c$. Now by \cite[Proposition 3.11]{Bhosle-2} we have $h^1(X, F) = 0$. By Lemma \ref{lem 2-3}, $F$ is a torsion-free sheaf of rank one and $\sigma$ is an inclusion. Let $E := \mathcal{H}om_{\mathcal{O}_X}(F, L)$. Then $E$ is a torsion-free sheaf of rank one on $X$. By \cite[Lemma 2.5]{Bhosle-1}, we have
\begin{equation}\label{E3}
\deg E  = 
\begin{cases}
s + 1, &\text{if both $F$ and $L$ are not locally free}\\
s, &\text{if either $F$ or $L$ is locally free}
\end{cases}
\end{equation}
Consider the parameter space
\begin{align*}
\mathcal{D}_{b,s} =  \{(F, F \hookrightarrow L, W \subset H^0(X,F), V \subset H^0(X,L)) \mid F \in \bar{J}^{d-s}(X),\\
(\phi: F \hookrightarrow L \in \mathbb{P}(H^0(X,E) ),W \in \text{Gr}(b,H^0(X,F)),\\ 
V \in \text{Gr}(c,H^0(X,L)),\phi_{|W}: W \hookrightarrow V \subset H^0(X,L)\}  
\end{align*}
where $\bar{J}^{d-s}(X)$ is the compactified Jacobian on $X$ parametrizing degree $d-s$ torsion-free sheaves on $X$. There are two projection maps
$$
\eta: \mathcal{D}_{b, s} \longrightarrow \bar{J}^{d-s}(X), \,\,\, (F, F \hookrightarrow L, W, V) \mapsto F
$$ 
and
$$
\rho: \mathcal{D}_{b, s} \longrightarrow \text{Gr}(c,H^0(X,L)), \,\,\, (F, F \hookrightarrow L, W, V) \mapsto V.
$$
Then $\dim \mathcal{D}_{b,s} \leq \dim \eta(\mathcal{D}_{b,s}) + \dim \eta^{-1}(F)$, for some $F \in \eta(\mathcal{D}_{b,s})$. The fiber of $\eta$ over $F \in \eta(\mathcal{D}_{b, s})$ has dimension (\cite[Theorem 2.7, Page 177]{Mistretta})
$$
\dim \mathbb{P}(H^0(X, \mathcal{H}om_{\mathcal{O}_X}(F, L))) \times \text{Gr}(b,H^0(X,F)) \times \text{Gr}(c,H^0(X,L)/W)
$$
and image $\eta(\mathcal{D}_{b,s})$ consists of all torsion-free sheaves $F \in \bar{J}^{d-s}(X)$ such that $h^0(X, \mathcal{H}om_{\mathcal{O}_X}(F, L)) \geq 1$. The space $\mathcal{D}_{b,s}$ parametrizes all the possible choices $(F, F \hookrightarrow L, W \subset H^0(X,F), V \subset H^0(X,L))$ for which there may exists a destabilizing subsheaf of $M_{V, L}^*$. We will consider two separate cases: when $E$ is a line bundle and when $E$ is not a line bundle.

\vspace{0.1 cm}

\underline{$E$ is not a line bundle:} By \cite[Chapter 8, Proposition 10]{Seshadri}, there is a line bundle $\mathcal{E}$ on $Y$ such that $E  = \pi_*(\mathcal{E})$ and $\deg \mathcal{E} = \deg E - 1$.

\underline{Case I:} Suppose both $F$ and $L$ are not locally free. By \eqref{E3}, $\deg \mathcal{E} = s$ and by Clifford's theorem  (applied on $\mathcal{E}$) yields
\begin{equation}\label{E4}
h^0(X, E) = h^0(Y, \mathcal{E})  \leq 
\begin{cases}
s/2  + 1 &\text{if} \,\, 0 \leq s \leq 2g - 4\\
s - g + 2 &\text{if} \,\, s \geq 2g - 3.
\end{cases}
\end{equation}
Since $h^0(X, E) \neq 0 \Longleftrightarrow h^0(Y, \mathcal{E}) \neq 0$ and
$$
\dim\{ M \in \text{Pic}^s(Y) \, : \, h^0(Y, M) \neq 0 \}  = \min (s, g-1)
$$
we get
\begin{align}\label{E5}
\begin{split}
\dim(\mathcal{D}_{b,s}) &\leq \min(s,g-1)  + \max(s-g +1, s/2)\\
&+ b(d-s-g+1-b) + c(s+b-c) \\
&\leq \frac{3}{2}s + b(d-s-g+1-b) + c(s+b-c)
\end{split}
\end{align}

\underline{Case II:} Suppose either $F$ and $L$ is locally free.  By \eqref{E3}, $\deg \mathcal{E} = s-1$ and by Clifford's theorem  (applied on $\mathcal{E}$)
\begin{equation}\label{E6}
h^0(X, E) = h^0(Y, \mathcal{E}) \leq 
\begin{cases}
(s-1)/2  + 1 &\text{if} \,\, 0 \leq s \leq 2g - 3\\
s - g + 1 &\text{if} \,\, s \geq 2g - 2.
\end{cases}
\end{equation}
Arguing as above, we get
\begin{align}\label{E7}
\begin{split}
\dim(\mathcal{D}_{b,s}) &\leq \min(s - 1,g-1) + \max(s-g, (s-1)/2)\\
&+ b(d-s-g+1-b) + c(s+b-c) \\
&\leq \frac{3}{2}s + b(d-s-g+1-b) + c(s+b-c)
\end{split}
\end{align}

\underline{$E$ is a line bundle:} By Lemma \ref{lem 3-1}, both $F$ and $L$ are locally free. By \eqref{E3}, $\deg E = s$. By Clifford theorem for line bundles over irreducible nodal curves (\cite[Theorem 3.1]{Karl}, \cite[Theorem A, Appendix by J. Harris]{Eisenbud}), we have
\begin{equation}\label{E8}
h^0(X, E) \leq 
\begin{cases}
s/2  + 1 &\text{if} \,\, 0 \leq s \leq 2g - 2\\
s - g + 1 &\text{if} \,\, s \geq 2g - 1.
\end{cases}
\end{equation}
Since $\dim \{ M \in \text{Pic}^s(X) \, : \, h^0(X, M) \geq 1 \} \leq \min( s, g)$, as above, we get
\begin{align}\label{E9}
\begin{split}
\dim(\mathcal{D}_{b,s}) &\leq \min(s, g) + \max(s-g, s/2)\\
&+ b(d-s-g+1-b) + c(s+b-c) \\
&\leq \frac{3}{2}s + b(d-s-g+1-b) + c(s+b-c)
\end{split}
\end{align}

The argument as given in \cite[Theorem 2.7, Page 177]{Mistretta} shows that
\begin{equation}\label{E10}
\frac{3}{2}s + b(d-s-g+1-b) + c(s+b-c) < c(d - g+ 1) - c^2
\end{equation}
For convenience, we will briefly recall the proof here: condition \eqref{E10} is equivalent to $\frac{3s}{2(c-b)} + s + b <  d+1-g$. From \eqref{E2}, we have
\begin{align*}
\begin{split}
\frac{3s}{2(c-b)} + s + b \leq \dfrac{\frac{3}{2}+ (c-b)}{g+c}d + b
\end{split}
\end{align*}
Notice that,
\begin{align*}
\dfrac{\frac{3}{2}+ (c-b)}{g+c}d + b < d + 1-g \Longleftrightarrow \dfrac{b+g-1}{b+g-\frac{3}{2}} < \frac{d}{g+c} \, .
\end{align*}
Since $b+g \geq 2, b + g - 1 \leq 2(b+g) - 3$. Moreover, $d/(g + c) > 2$. Thus the last inequality holds. Hence
$$
\dim(\mathcal{D}_{b,s}) < \dim \text{Gr}(c, H^0(X, L)).
$$
Then for a generic choice of $V \in \text{Gr}(c,H^0(X,L))$, no such destabilizing data $(F, F \hookrightarrow L, W \subset H^0(X,F), V \subset H^0(X,L)) \in \mathcal{D}_{b,s}$ exists and hence $M_{V, L}^*$ is stable.
\end{proof}

\begin{Theorem}\label{thm 3-2}
Let $X$ be an integral projective curve of arithmetic genus $g \geq 3$ with an ordinary double point  and let $L$ be a globally generated torsion-free sheaf of rank one and degree $d = 2g + 2c, 1 \leq c \leq g$. Then for a generic subspace $V \subseteq H^0(X, L)$ of codimension $c$ that generates $L$, $M_{V, L}^*$ is semistable. Moreover, if $X$ is non-hyperelliptic, then $M_{V, L}^*$ is stable.
\end{Theorem}
\begin{proof}
We proceed as in Theorem \ref{thm 3-1}. Let $N$ be a stable sub-bundle of $M_{V, L}^*$ of maximal slope such that $\mu(N) > \mu(M_{V, L}^*) = -2$. Then Lemma \ref{lem 2-2}, Lemma \ref{lem 2-3} and \cite[Proposition 3.11]{Bhosle-2}, we have a commutative diagram
\begin{center}
 \begin{tikzpicture}
  \matrix(m)[matrix of math nodes,row sep=2.0em, column sep=3.5em,text height=1.5ex, 
 text depth=0.25ex]
 {  0 &  M_{V, L}^* & V \otimes \mathcal{O}_X & L & 0\\
 0 & N & W \otimes \mathcal{O}_X & F & 0 \\};
 \path[->](m-1-1) edge (m-1-2);
 \path[->](m-1-2) edge  (m-1-3);
 \path[->](m-1-3) edge  (m-1-4);
 \path[->](m-1-4) edge  (m-1-5);
 \path[->](m-2-1) edge (m-2-2);
 \path[->](m-2-2) edge  (m-2-3);
 \path[->](m-2-3) edge  (m-2-4);
 \path[->](m-2-4) edge  (m-2-5);
 \path[right hook ->](m-2-2) edge  (m-1-2);
 \path[right hook ->](m-2-3) edge  (m-1-3);
 \path[->](m-2-4) edge node[right]{$\sigma$}  (m-1-4);
 \end{tikzpicture}
 \end{center}
where $F$ is torsion-free subsheaf of $L$, $\sigma$ is injective and $h^1(X, F) = 0$. Arguing as in Theorem \ref{thm 3-1}, we get that for a generic choice of $V \subseteq H^0(X, L)$ of codimension $c$, no such $N$ exists. Thus $M_{V, L}^*$ is semistable.

\vspace{0.1 cm}

Now let $N$ be a stable sub-bundle of $M_{V, L}^*$ of maximal slope with $\mu(N) = \mu(M_{V, L}^*) = -2$ and consider the diagram
\begin{center}
 \begin{tikzpicture}
  \matrix(m)[matrix of math nodes,row sep=2.0em, column sep=3.5em,text height=1.5ex, 
 text depth=0.25ex]
 {  0 &  M_{V, L}^* & V \otimes \mathcal{O}_X & L & 0\\
 0 & N & W \otimes \mathcal{O}_X & F & 0 \\};
 \path[->](m-1-1) edge (m-1-2);
 \path[->](m-1-2) edge  (m-1-3);
 \path[->](m-1-3) edge  (m-1-4);
 \path[->](m-1-4) edge  (m-1-5);
 \path[->](m-2-1) edge (m-2-2);
 \path[->](m-2-2) edge  (m-2-3);
 \path[->](m-2-3) edge  (m-2-4);
 \path[->](m-2-4) edge  (m-2-5);
 \path[right hook ->](m-2-2) edge  (m-1-2);
 \path[right hook ->](m-2-3) edge  (m-1-3);
 \path[->](m-2-4) edge node[right]{$\sigma$}  (m-1-4);
 \end{tikzpicture}
 \end{center}
as obtained in Lemma \ref{lem 2-2}. If $h^1(X, F) = 0$, then we can use the above argument to conclude the stability of $M_{V, L}^*$. So assume now $h^1(X, F) \neq 0$. Since $X$ is non-hyperelliptic,  by \cite[Proposition 3.11]{Bhosle-2}, we get $F = K_X$, $W = H^0(X, K_X)$ and $N = M_{H^0(X, K_X), K_X}$. In this case,
$$
\mathcal{D}_{b,s} = \{ (K_X \hookrightarrow L, V \subseteq H^0(X, L)) \,  : \, H^0(X, K_X) \subseteq V \}.
$$
Let $E  = \mathcal{H}om_{\mathcal{O}_X}(F, L)$. Then by \cite[Lemma 2.5]{Bhosle-1} $\deg E = 2c+ 2$. 
If $L$ is not locally free, then by Lemma \ref{lem 3-1}, $E$ is not locally free. In this case, similar to the case of \eqref{E6}, we get
\begin{equation}\label{E11}
h^0(X, E) \leq \max \{ (2c+1)/2 + 1, 2c - g + 3 \}.
\end{equation}
If $L$ is locally free, then arguing as in \eqref{E8}, we get
\begin{equation}\label{E12}
h^0(X, E) \leq \max \{ c+2, \, 2c - g + 3 \}.
\end{equation}
Since $\dim H^0(X, L)/H^0(X, K_X) = 2c + 1$, we have
\begin{equation}\label{E13}
\dim \mathcal{D}_{b, s} \leq \max \{ c+1, 2c - g + 2 \} + c(c + 1).
\end{equation}
Since $g \geq 3$, we have
\begin{align*}
c + 1 < gc \,\,\, \text{and} \,\,\, 2c - g + 2 \leq 2c - 3 + 2 < gc - 1 < gc.
\end{align*}
Thus $\dim \mathcal{D}_{b, s} < gc + c(c+1) = \dim \text{Gr}(c, H^0(X, L))$. Hence for a generic $V \in \text{Gr}(c, H^0(X, L))$, $M_{V, L}^*$ is stable.
\end{proof}

\section{Stability of Kernel Sheaf: Multi Node}

Let $X$ be an integral projective curve of arithmetic genus $g \geq 2$ with $n$-nodes (ordinary double points), say $x_1, \cdots, x_n$. Let $L$ be a torsion-free sheaf of rank one on $X$. We say $L$ has {\it local type one} at $x_i$ if $L_{x_i} \cong \mathfrak{m}_i$ where $\mathfrak{m}_i \subseteq \mathcal{O}_{X, x_i}$ be the unique maximal ideal and $L$ has {\it local type zero} at $x_i$ if $L_{x_i} \cong \mathcal{O}_{X, x_i}$. Let $\pi: Y \longrightarrow X$ be the normalization. Then the genus of $Y$ is $g_Y =  g - n$.


\begin{Theorem}\label{thm 4-1}
Let $X$ be an integral projective curve of arithmetic genus $g \geq 3$ with $n$-nodes. Let $L$ be a globally generated torsion-free sheaf of rank one and degree $d > 10(g+c)/3, 1 \leq c \leq g$. Then for a generic subspace $V \subseteq H^0(X, L)$ of codimension $c$ that generates $L$, $M_{V, L}^*$ is stable. Moreover, if the number of node $n \leq g - 2$, then $M_{V, L}^*$ is stable for $d > 2g + 2c$.
\end{Theorem}
\begin{proof}
Let $N$ be a stable sub-bundle of $M_{V, L}^*$ of maximal slope By Lemma \ref{lem 2-2} and Lemma \ref{lem 2-3}, there is a commutative diagram
\begin{center}
 \begin{tikzpicture}
  \matrix(m)[matrix of math nodes,row sep=2.0em, column sep=3.5em,text height=1.5ex, 
 text depth=0.25ex]
 {  0 &  M_{V, L}^* & V \otimes \mathcal{O}_X & L & 0\\
 0 & N & W \otimes \mathcal{O}_X & F & 0 \\};
 \path[->](m-1-1) edge (m-1-2);
 \path[->](m-1-2) edge  (m-1-3);
 \path[->](m-1-3) edge  (m-1-4);
 \path[->](m-1-4) edge  (m-1-5);
 \path[->](m-2-1) edge (m-2-2);
 \path[->](m-2-2) edge  (m-2-3);
 \path[->](m-2-3) edge  (m-2-4);
 \path[->](m-2-4) edge  (m-2-5);
 \path[right hook ->](m-2-2) edge  (m-1-2);
 \path[right hook ->](m-2-3) edge  (m-1-3);
 \path[->](m-2-4) edge node[right]{$\sigma$}  (m-1-4);
 \end{tikzpicture}
 \end{center}
where $F$ is a rank one torsion-free sheaf and $\sigma$ is an inclusion.  Since $\mu(N) \geq \mu(M_{V, L}^*) > -2$, by \cite[Proposition 3.11]{Bhosle-2}, $h^1(X, F) = 0$. Let $x_1, \cdots, x_n \in X$ be the nodes of $X$ and assume $L$ has local type one at $x_1, \cdots, x_r, r \leq n$ and local type zero at the rest of the nodes. Set $E := \mathcal{H}om_{\mathcal{O}_X}(F, L)$. Then $E$ is a torsion-free sheaf of rank one. By Lemma \ref{lem 3-1}, $E$ has local type one at $x_1, \cdots, x_r$. After renumbering, if necessary, we may assume that $E$ has local type one at $x_{r+1}, \cdots, x_{r + t}$ and has local type zero at the rest of the nodes. By \cite[Lemma 2.5]{Bhosle-1}, $s \leq \deg E \leq s + r$ where $s := \deg L - \deg F$. If $F$ has local type zero at any of the points $x_1, \cdots, x_r$, then $\deg E < s + r$. There is a short exact sequence
$$
0 \longrightarrow \mathcal{F} \longrightarrow E \longrightarrow \oplus_{j=r + t + 1}^n \mathbb{C}(x_j) \longrightarrow 0
$$
where $\mathbb{C}(x_j) = \mathbb{C}$, for every $j = r + t + 1, \cdots, n$, $\mathcal{F}$ is a torsion-free sheaf of rank one and $\deg \mathcal{F} = \deg E - n + (r+t)$. Moreover, $\mathcal{F}$ has local type one at every node $x_1, \cdots, x_n$. By \cite[Chapter 8, Proposition 10]{Seshadri}, there is a line bundle $V$ on $Y$ such that $\pi_*\mathcal{V} = \mathcal{F}$ and $\deg \mathcal{V} = \deg \mathcal{F} - n = \deg E - 2n + (r+t)$. By Clifford's theorem, if $0 \leq \deg \mathcal{V} \leq 2g_Y - 2$, then
$$
h^0(Y, \mathcal{V}) \leq \dfrac{\deg \mathcal{V}}{2} + 1 = \dfrac{\deg E - (r+t)}{2} + 1
$$
and if $\deg \mathcal{V} \geq 2g_Y - 1$, then $h^0(Y, \mathcal{V}) = \deg \mathcal{V} + 1 - g_Y = \deg E - g + 1 - n + (r+t)$. Since $h^0(X, \mathcal{F}) = h^0(X, \mathcal{V})$ and $\deg E - (r+t) \leq s$, we have
\begin{align}\label{E14}
h^0(X, E) \leq h^0(Y, \mathcal{V}) + n - (r+t) \leq \max \{ s/2 + 1, \, s + r - g + 1 \}.
\end{align}

Consider the parameter space $\mathcal{D}_{b, s}$ as in Theorem \ref{thm 3-1} and the projection morphism $\eta: \mathcal{D}_{b, s} \longrightarrow \bar{J}^{d-s}(X)$. The image of $\eta$ consists of torsion-free sheaves of rank one and degree $d - s$ such that $h^0(X, \mathcal{H}om_{\mathcal{O}_X}(F, L)) \neq 0$. Suppose $E := \mathcal{H}om_{\mathcal{O}_X}(F, L)$ has local type one at $m \leq n$ points and local type zero at the rest of the points. Let $\pi': Y' \longrightarrow X$ be the blow-up of $X$ at exactly the above $m$ points. Then there is a line bundle $\mathcal{M}$ on $Y'$ of degree $\deg \mathcal{M} = \deg E - m$ such that $\pi_*'\mathcal{M} = E$. Since $h^0(X, E) \neq 0$ if and only if $h^0(Y', \mathcal{M}) \neq 0$ and $\dim \{ \mathcal{N} \in \text{Pic}^{\deg \mathcal{M}}(Y') \, : \, h^0(Y', \mathcal{N}) \neq 0 \} = \min ( \deg \mathcal{M}, g - m) = \min ( \deg E - m, g -m )$, $\dim \eta(\mathcal{D}_{b, s})$ is bounded above by $\max_{0 \leq m \leq n} \min (\deg E - m, g -m ) = \min (\deg E, g) \leq \min (s + r, g)$.

\vspace{0.1 cm}

Now we proceed as in Theorem \ref{thm 3-1}:
\begin{align}\label{E15}
\begin{split}
\dim D_{b,s} &\leq \min \{s+r, g \}   + \max \{ s/2, s + r -g \}\\
&+  b(d-s-g+1-b) + c(s+b-c) \\
&\leq 3\frac{s}{2} + r + b(d-s-g+1-b) + c(s+b -c) \\
&\leq 3\frac{s}{2} + n + b(d-s-g+1-b) + c(s+b -c).
\end{split}
\end{align}
and we want to show
\begin{align}\label{E16}
\begin{split}
3\frac{s}{2} + n + &b(d-s-g+1-b) + c(s+b -c) < c(d+1 -g) - c^2 \\
 &\Longleftrightarrow 3\frac{s}{2} +n < (c-b)(d+1-g -s-b) \\
 &\Longleftrightarrow \frac{3s}{2(c-b)} + \frac{n}{c-b} + s +b < d +1 -g
\end{split}
\end{align}
By Equation \eqref{E2}, we have $c - b \geq 1$ so that
\begin{equation}\label{E17}
\dfrac{3s}{2(c-b)} + \dfrac{n}{c-b} + s +b \leq \dfrac{\frac{3}{2}+ (c-b)}{g+c} d +   n  + b
\end{equation}
So it is sufficient to prove
\begin{align}\label{E18}
\begin{split}
\dfrac{\frac{3}{2}+ (c-b)}{g+c}d +   n  + b  < d+1-g  \Longleftrightarrow \dfrac{b+g+n -1 }{b+g-\frac{3}{2}} < \dfrac{d}{g+c}.
\end{split}
\end{align}
Since $b + g \geq 2$, we have
\begin{align*}
b + 2g - 3 \leq 2(b + g) - 3 &\Longrightarrow \dfrac{b + 2g - 3}{b+g-\frac{3}{2}} \leq 2
\end{align*}
and hence
\begin{align}\label{E19}
\dfrac{b+g+n -1 }{b+g-\frac{3}{2}} \leq \dfrac{(b+2g - 3) + (n + 2 -g)}{b+g-\frac{3}{2}} \leq 2 + \dfrac{n + 2 -g}{b+g-\frac{3}{2}} .
\end{align}
Since $n \leq g$, we have $n - g + 2 \leq 2$. Moreover, $b + g - \frac{3}{2} \geq \frac{3}{2}$. Thus from equation \eqref{E19}, we get
\begin{align}\label{E20}
\dfrac{b+g+n -1 }{b+g-\frac{3}{2}} \leq 2 + \dfrac{4}{3} = \dfrac{10}{3} \, .
\end{align}
Since $d > 10(g+c)/3$, equation \eqref{E18} holds. If $n \leq g - 2$, then $n + 2 - g \leq 0$. So from equation \eqref{E20}, we get
\begin{align*}
\dfrac{b+g+n -1 }{b+g-\frac{3}{2}} \leq 2.
\end{align*}
In this case, for $d > 2g + 2c$, then equation \eqref{E18} holds. Thus, in either cases, we get
$$
\deg \mathcal{D}_{b,s} < \dim \text{Gr}(c, H^0(X, L))
$$
and hence for a generic choice of $V \in \text{Gr}(c, H^0(X, L))$, $M_{V, L}^*$ is stable.
\end{proof}

\begin{Remark}
Let $X$ be an integral projective curve of arithmetic genus $g \geq 2$ with $n$ nodes and let $L$ be a globally generated torsion free sheaf of rank one and degree $d > 2g$. Let $V \subseteq H^0(X, L)$ be a subspace of codimension $c, 1 \leq c\leq g$. Then $h^1(X, L) = 0$ and by Riemann-Roch, we have
$$
\dim V = h^0(X, L) - c = d + 1 - g - c \geq d + 1 - 2g.
$$
By Proposition \cite[Proposition 2.3]{Bhosle-2}, $L$ is generated by two generic sections. In each of the cases of Theorem \ref{thm 3-1}, Theorem \ref{thm 3-2} and Theorem \ref{thm 4-1}, the condition on degree $d$ implies $\dim V \geq 2$. Thus, a generic subspace $V \subseteq H^0(X, L)$ of codimension $c$ generates the torsion free sheaf $L$.
\end{Remark}

\subsection*{Acknowledgements} We thank the reviewer for useful comments and suggestions to make the exposition better.


\begin{thebibliography}{100}
\bibitem{Bhosle-1}  Bhosle U.N.;  Maximal subsheaves of torsion-free sheaves on nodal curves. {\it  J. Lond. Math. Soc.} 74, 59-74 (2006).

\bibitem{Bhosle-2}  Bhosle, Usha N.; Singh, Sanjay Kumar; Brill-Noether loci and generated torsionfree sheaves over nodal and cuspidal curves. {\it Manuscr. Math.} 141, No. 1-2, 241-271 (2013).


\bibitem{Bhosle-3} Bhosle, U. N.; Brambila-Paz, L.; Newstead, P. E.; On linear series and a conjecture of D. C. Butler. {\it Int. J. Math.} 26, No. 2, Article ID 1550007, 18 p. (2015).




\bibitem{B-F}Brivio, Sonia; Favale, Filippo F.; On kernel bundles over reducible curves with a node. {\it Int. J. Math.} 31, No. 7, Article ID 2050054, 15 p. (2020). 


\bibitem{Butler}  Butler, David C.;  Normal generation of vector bundles over a curve. {\it J. Differ. Geom.} 39, No. 1, 1-34 (1994).


\bibitem{Butler-2} Butler, David C.; Birational maps of moduli of Brill-Noether pairs, preprint (1997), arXiv:alg-geom/9705009.


\bibitem{Karl} Christ, Karl; A Clifford inequality for semistable curves. {\it Math. Z.} 303, No. 1, Paper No. 15, 20 p. (2023).

\bibitem{E-L-M} Ein, Lawrence; Lazarsfeld, Robert; Mustopa, Yusuf; Stability of syzygy bundles on an algebraic surface. {\it Math. Res. Lett.} 20, No. 1, 73-80 (2013). 


\bibitem{Eisenbud} Eisenbud, David; Koh, Jee; Stillman, Michael; Determinantal equations for curves of high degree. {\it Am. J. Math.} 110, No. 3, 513-539 (1988).

 \bibitem{Mistretta}  Mistretta, Ernesto C.; Stability of line bundle transforms on curves with respect to low codimensional subspaces. {\it J. Lond. Math. Soc., II. Ser.} 78, No. 1, 172-182 (2008).
 

\bibitem{P-R} Paranjape, Kapil; Ramanan, S.; On the canonical ring of a curve. {\it Algebraic geometry and commutative algebra}, in Honor of Masayoshi Nagata, Vol. II, 503-516 (1988).


\bibitem{R} Rekuski, Nick; Stability of kernel sheaves associated to rank one torsion-free sheaves. {\it Math. Z.} 307, No. 1, Paper No. 2, 18 p. (2024). 
 
 
\bibitem{Seshadri} Seshadri C.S.; Fibr\'es vectoriels sur les courbes alg\'ebriques. {\it Astérisque.} 96 (1982).


 
 
 
 
\end{thebibliography}
\end{document}